\documentclass[12pt,letterpaper,twoside]{article}

\usepackage{amsmath,amssymb}
\usepackage{amsthm}
\usepackage{mathtools}
\usepackage{stmaryrd}
\usepackage{fancyhdr}
\usepackage{times}
\usepackage{mathrsfs} 
\usepackage{titlesec}
\usepackage{enumitem}
\usepackage[pdftex,dvips]{geometry}
\usepackage[bottom,multiple,hang]{footmisc}

\input diagxy.tex

\mathtoolsset{centercolon,mathic}

\titlelabel{\thetitle.\ }
\titleformat*{\section}{\normalsize\bfseries\centering}
\titleformat*{\subsection}{\normalsize\itshape}

\newtheoremstyle{fact}
     {\topsep}
     {\topsep}
     {\slshape}
     {}
     {\bfseries}
     {}
     { }
     {\thmname{#1}\thmnumber{ #2.}\thmnote{ \rm (#3)}}

\newtheoremstyle{mylabel} 
        {\topsep}
        {\topsep}
        {\itshape}
        {}
        {\bfseries}
        {}
        { }
        {\thmname{#1}\thmnote{ #3}.} 

\newtheorem{theorem}{Theorem}[section]
\newtheorem{Ltheorem}{Theorem}

\newtheorem*{theorem*}{Theorem} 
\newtheorem{lemma}[theorem]{Lemma}
\newtheorem{proposition}[theorem]{Proposition}
\newtheorem{corollary}[theorem]{Corollary}

\theoremstyle{definition}
\newtheorem{definition}[theorem]{Definition}

\newtheorem*{remark*}{Remark}

\newtheorem*{question*}{Question}
\newtheorem*{examples*}{Examples}  

\newtheorem*{example*}{Example}

\theoremstyle{mylabel}
\newtheorem*{Ltheorem*}{Theorem}

\theoremstyle{fact}

\setenumerate{topsep=0pt,labelwidth=0pt,align=left,labelindent=0pt,
labelsep=0.0in,labelwidth=0.275in,leftmargin=0.275in,label=\rm(\alph*),
parsep=0pt,itemsep=3pt}

\def\proofont{\fontseries{bx}\fontshape{sc}\selectfont}
\def\proofname{Proof. }

\newcommand{\homeo}{\operatorname{Homeo}}
\newcommand{\supp}{\operatorname{supp}}

\makeatletter
\renewenvironment{proof}[1][\proofname]{\par
  \normalfont
  \topsep6\p@\@plus6\p@ \trivlist
  \item[\hskip\labelsep\noindent\proofont #1]\ignorespaces
}{%
  \qed\endtrivlist
}
\makeatother

\makeatletter
\def\@fnsymbol#1{\ifcase#1\or * \or 1 \or 2  \else\@ctrerr\fi\relax}
\let\mytitle\@title

\author{Rafael Dahmen and G\'abor Luk\'acs}

\title{Long colimits of topological groups IV: Spaces with socks}

\begin{document}

\makeatletter
\let\mytitle\@title
\chead{\small\itshape R. Dahmen and G. Luk\'acs / Long colimits of topological groups IV}
\fancyhead[RO,LE]{\small \thepage}
\makeatother

\maketitle

\def\thanks#1{}

\thispagestyle{empty}

\begin{abstract} 
The group of compactly supported homeomorphisms on a Tychonoff  space can be topologized
in a number of ways, including as a colimit of homeomorphism groups 
with a given compact support, or as a subgroup of the homeomorphism group of its Stone-\v{C}ech compactification. A space is said to have the {\itshape Compactly Supported Homeomorphism Property} ({\itshape CSHP})
if these two topologies coincide. The authors develop techniques for showing
that products of certain spaces with CSHP, such as the Closed Long Ray and the Long Line,
have CSHP again.
\end{abstract}

\section{Introduction}

Given a compact space $K$, it is well known that the homeomorphism 
group $\homeo(K)$ is a topological group with the compact-open topology (\cite{Arens}).
If $X$ is assumed to be only Tychonoff, then 
for every compact subset $K\subseteq X$,
the group $\homeo_K(X)$ of homeomorphisms supported in $K$ (i.e., identity on $X\backslash K$) is  a topological
group with the compact-open topology; however,
the full homeomorphism group $\homeo(X)$ equipped
with the compact-open topology need not be a topological group (\cite{dijkstra}).
Nevertheless, $\homeo(X)$ can be turned into
a topological group by embedding it into $\homeo(\beta X)$, 
the homeomorphism group of the Stone-\v{C}ech compactification of $X$. 
The latter topology has also been studied under the name of {\itshape zero-cozero topology}
(\cite{osipov2020}, \cite{diConcilio}).

For a Tychonoff space $X$, let $\mathscr{K}(X)$ denote the poset of compact subsets of $X$
ordered by inclusion.
In light of the foregoing, 
the group $\homeo_{cpt}(X)\coloneqq\bigcup\limits_{K\in \mathscr{K}(X)}\homeo_K(X)$
of the compactly supported homeomorphisms of $X$ admits seemingly different topologies:

\begin{enumerate}

\item 
the finest topology making all inclusions $\homeo_K (X) \to \homeo_{cpt}(X)$ continuous
(i.e., the colimit in the category of topological spaces and continuous maps); and

\item
the topology induced by $\homeo(\beta X)$.

\end{enumerate}

A space $X$ is said to have the {\itshape Compactly Supported Homeomorphism Property} 
({\itshape CSHP}) if these topologies coincide (\cite{RDGL1}, \cite{RDGL3}). In a previous
work, the authors gave a complete characterization of a finite product
of ordinals having CSHP (\cite[Theorem~A]{RDGL3}). Having CSHP is not productive, though.
For example, while $\omega_1$ and $\omega_2$ equipped with the order
topology have CSHP (\cite[Theorem~D(c)]{RDGL1}), their product $\omega_1 \times \omega_2$ does not
(\cite[Example~1]{RDGL3}).

In this paper, we develop techniques  allowing one to show that products of
certain spaces with CSHP have CSHP again. As an application, we prove the following result.

\begin{Ltheorem} \label{main:thm:product}
Let $C$ be a compact metrizable space, let $\mathbb{L}_{\geq 0}$ denote the Closed Long Ray,
and $\mathbb{L}$ denote the Long Line. For every $n,m,l<\omega$,
the product space 
$\mathbb{L}^n \times (\mathbb{L}_{\geq 0})^m \times (\omega_1)^l \times C$
has CSHP.  
\end{Ltheorem}

For the definition of $\mathbb{L}_{\geq 0}$ and $\mathbb{L}$, see Section~\ref{sect:basic}.
The proof of Theorem~\ref{main:thm:product} requires introducing the notion 
of a space having a {\itshape $\kappa$-sock},
which is productive (Proposition~\ref{basic:prop:sock-prod}(b)), and under suitable conditions, 
implies CSHP. To that end, we recall some terminology and notation.

Given a  cardinal $\kappa$, a directed set $(\mathbb{I},\leq)$ is {\em $\kappa$-long}  if every subset $C\subseteq \mathbb{I}$ with $|C| \leq \kappa$ has an upper bound in $\mathbb{I}$. One says that $\mathbb{I}$ is {\em long} if it is $\aleph_0$-long.  Recall that a poset $(\mathbb{I},\leq)$
is an {\em $\omega$-cpo} if every non-decreasing sequence $\{\alpha_k\}_{k < \omega}$
in $\mathbb{I}$ has a supremum. 

All spaces in this paper are assumed to be Tychonoff. We denote by
$\mathscr{C}(X,Y)$ the space of continuous functions from a space $X$ to a space $Y$,
equipped with the compact-open topology.

\begin{definition}
Let $\kappa$ be an infinite cardinal. A {\itshape $\kappa$-sock} on a space $X$
is a family $\{p_\alpha\}_{\alpha \in \mathbb{I}}$
of retracts on $X$, indexed by a $\kappa$-long $\omega$-cpo $(\mathbb{I},\leq)$ satisfying
the following properties:

\begin{enumerate}[label=(S\arabic*),labelwidth=0.4in,leftmargin=0.4in]

\item 
$K_\alpha \coloneqq p_\alpha(X)$ is compact
for every $\alpha \in \mathbb{I}$;

\item
for every $\alpha,\beta \in \mathbb{I}$ such that $\alpha \leq \beta$, one has
$p_\beta p_\alpha = p_\alpha$ (or equivalently, $K_\alpha \subseteq K_\beta$);

\item
$\{K_\alpha\}_{\alpha \in \mathbb{I}}$ is cofinal
in the family $\mathscr{K}(X)$ of compact subsets
of $X$ (ordered by inclusion);

\item
for every non-decreasing sequence $\{\alpha_k\}_{k <\omega}$ 
with $\gamma = \sup\limits_{k<\omega} \alpha_k$ in $\mathbb{I}$, one has
$\lim p_{\alpha_k} = p_{\gamma}$
in $\mathscr{C}(X,K_{\gamma})$.

\end{enumerate}

\noindent
In the special case where $\kappa=\aleph_0$, we say that $\{p_\alpha\}_{\alpha \in \mathbb{I}}$
is a {\itshape sock}. Note that if $\kappa_1 \geq \kappa_2$ are infinite cardinals, then
every $\kappa_1$-sock is also a $\kappa_2$-sock. Furthermore, if $X$ has a sock, then
$(\mathscr{K}(X),\subseteq)$ is long, and in particular, every countable subset of $X$
is contained in a compact subset, and so $X$ is pseudocompact.
\end{definition}

Recall that the \emph{tightness} $t(X,\mathscr{S})$ of a topological space $(X,\mathscr{S})$ is the smallest cardinal $\kappa$ such that every point $p$ in the closure of a subset $A\subseteq X$ is in the closure of a subset $C\subseteq A$ with $|C| \leq \kappa$. One says that a space is \emph{$\kappa$-tight} if $t(X,\mathscr{S})\leq\kappa$. Recall further that the {\itshape compact weight}
of a space $X$ is $kw(X)\coloneqq \sup \{ w(K) \mid K \in \mathscr{K}(X)\}$, where
$w(K)$ is the weight of $K$
(\cite{hernandezEtAl2014}, \cite{hernandezEtAl2020}). (Following \cite[1.7.12]{Engel},
if $K$ is finite, one puts $kw(K)=w(K)=\aleph_0$.)

\begin{Ltheorem} \label{main:thm:sock}
Let $\kappa$ be an infinite cardinal. Suppose that $X$ is a locally
compact space that has a $\kappa$-sock  and $kw(X) \leq \kappa$. Then $X$ has CSHP and
$\homeo_{cpt}(X)$ is $\kappa$-tight.
\end{Ltheorem}

The paper is structured as follows. In Section~\ref{sect:basic}, we present
basic properties of spaces with $\kappa$-socks, such as preservation under
finite disjoint unions, arbitrary products, and certain quotients. We also
show that a continuous map from a space with a sock into a metrizable space is
determined by its values on a given compact subset (Theorem~\ref{basic:thm:evalpha}).
In Section~\ref{sect:socks}, we prove Theorems~\ref{main:thm:product} and~\ref{main:thm:sock}.

\section{Basic Properties}

\label{sect:basic}


\begin{proposition} \label{basic:prop:totallyordered}
Let $\kappa$ be an infinite cardinal, and
let $(\mathbb{I},\leq)$ be a $\kappa$-long
totally ordered set in which every subset bounded from above has a supremum in $\mathbb{I}$.
Let $X$ denote $\mathbb{I}$ with the order topology. 
Then $X$ has a $\kappa$-sock indexed by $(\mathbb{I},\leq)$.
\end{proposition}

\begin{proof}
Since the empty set is bounded, it has a supremum in $\mathbb{I}$, and so
$\mathbb{I}$ has a minimum. Put $0\coloneqq \min \mathbb{I}$.
Since $(\mathbb{I},\leq)$ is $\kappa$-long, every countable subset
has an upper bound; consequently, by our assumption, every countable
subset has a supremum. Thus, $(\mathbb{I},\leq)$ is an $\omega$-cpo.
For every $\alpha \in \mathbb{I}$, put $p_\alpha(x) = \min(x,\alpha)$. 
We show that $\{p_\alpha\}_{\alpha \in \mathbb{I}}$ is a $\kappa$-sock on $X$.

(S1) One has $p_\alpha(X) = [0,\alpha]$, which is compact (\cite[3.12.3(a)]{Engel}).

(S2) Since $\min(\min(x,\alpha),\beta) = \min(x,\min(\alpha,\beta))$, it follows
that $p_\beta p_\alpha = p_\alpha$ whenever $\alpha \leq \beta$.

(S3)  If $K \subseteq X$ is compact, then $K$ has an upper bound, say $\alpha \in \mathbb{I}$,
and so $K \subseteq [0,\alpha] = K_\alpha$.

(S4) It follows from (S1) that $X$ is locally compact. 
The map  $\min\colon X \times X \rightarrow X$ is continuous. Thus,
$p\colon X \rightarrow \mathscr{C}(X,X)$, defined
by $p(\alpha)=p_\alpha$, is continuous (cf.~\cite[3.4.8]{Engel}).
Therefore, if $\{\alpha_k\}_{k < \omega}$ is a non-decreasing sequence in
$\mathbb{I}$, then $p(\sup \alpha_k) = p( \lim \alpha_k) = \lim p(\alpha_k)$.
\end{proof}

\begin{corollary} \label{basic:cor:delta}
Let $\kappa$ be an infinite cardinal, and let $\delta$ be an ordinal.
If $\delta$ is $\kappa$-long, then the space $X=\delta$ equipped with the order topology
has a $\kappa$-sock. In particular, $\omega_1$ has a sock. \qed
\end{corollary}

The product $\mathbb{L}_{\geq 0}:=\omega_1 \times [0,1)$ equipped with order topology generated by the lexicographic order is called the {\itshape Closed Long Ray}. We identify
$\omega_1$ with the closed cofinal subset $\omega_1 \times \{0\}$ of $\mathbb{L}_{\geq 0}$.

\begin{corollary} \label{basic:cor:longray}
The Closed Long Ray $\mathbb{L}_{\geq 0}$ with the order topology has a sock and
has a countable compact weight.
\end{corollary}

\begin{proof}
Since $(\omega_1,\leq)$ is a long poset (with the lexicographic order), so is  $(\mathbb{L}_{\geq 0},\leq)$. Every interval $[0,x]$ in $(\mathbb{L}_{\geq 0},\leq)$ in the order topology
is homeomorphic to the compact interval $[0,1]$ (\cite[1.10]{GauldBook}).
(In particular, every compact subset of $\mathbb{L}_{\geq 0}$  has a countable weight.)
Therefore,
every subset bounded from above has a supremum in $\mathbb{L}_{\geq 0}$, 
and the statement follows by Proposition~\ref{basic:prop:totallyordered}.
\end{proof}

\begin{proposition}
Let $\kappa$ be an infinite cardinal, \label{basic:prop:sock-coprod}
let $\{X_j\}_{j \in J}$ be a family
of  topological spaces, and put
$X \coloneqq \coprod\limits_{j\in J} X_j$. Then:

\begin{enumerate}

\item 
$kw(X) = \sup\limits_{j \in J} kw(X_j)$; and

\item
if $J$ is finite and $X_j$ has a $\kappa$-sock for every $j \in J$, then 
$X$ has a $\kappa$-sock.

\end{enumerate}
\end{proposition}

\begin{proof}
(a) Clearly, $kw(X_j) \leq kw(X)$ for every $j \in J$, and so
$\sup\limits_{j \in J} kw(X_j) \leq kw(X)$. We show now the reverse inequality.
Let $K \subseteq X$ be compact. Then there is a finite subset $F\subseteq J$ such that
$K \subseteq \coprod\limits_{j\in F} X_j$, and $K = \coprod\limits_{j\in F} K \cap X_j$.
Therefore,
\begin{align}
    w(K) \leq \prod\limits_{j \in F} w(K \cap X_j) \leq \sup\limits_{j \in J} kw(X_j).
\end{align}

(b) For every $j \in J$, let $\{p_\alpha^{(j)}\}_{\alpha \in \mathbb{I}_j}$ be a $\kappa$-sock on\ $X_j$. Put $\mathbb{I} \coloneqq \prod\limits_{j \in J} \mathbb{I}_j$ equipped
with the coordinatewise order, that is, $(\alpha^{(j)}) \leq (\beta^{(j)})$ if 
$\alpha^{(j)} \leq \beta^{(j)}$ for every $j \in J$. Then $(\mathbb{I},\leq)$ is a 
$\kappa$-long $\omega$-cpo, because upper bounds and suprema can be calculated 
coordinatewise. For $\alpha=(\alpha^{(j)})\in\mathbb{I}$, put
$p_\alpha\coloneqq \coprod\limits_{j \in J} p^{(j)}_{\alpha^{(j)}}$. We verify that 
$\{p_\alpha\}_{\alpha \in \mathbb{I}}$ satisfies the conditions of a $\kappa$-sock on $X$.

(S1) For every $(\alpha^{(j)})\in\mathbb{I}$, the image
$K_\alpha\coloneqq p_\alpha(X)  =  
\coprod\limits_{j \in J} p^{(j)}_{\alpha^{(j)}} (X_j)$ is compact,
because $p^{(j)}_{\alpha^{(j)}} (X_j)$ is compact for every $j \in J$ and $J$ is finite.

(S2) If  $\alpha=(\alpha^{(j)}) \leq \beta=(\beta^{(j)})$ in $\mathbb{I}$, then
$\alpha^{(j)} \leq \beta^{(j)}$ for every $j \in J$, and so 
$p^{(j)}_{\beta^{(j)}} p^{(j)}_{\alpha^{(j)}} = p^{(j)}_{\alpha^{(j)}}$ for every $j \in J$.
Thus, $p_\beta p_\alpha = p_\alpha$.

(S3) If $K \subseteq X$ is compact, then $K \cap X_j$ is compact in $X_j$ for 
every $j \in J$, and so there is $\alpha^{(j)} \in \mathbb{I}_j$ such that
$K \cap X_j \subseteq K_{\alpha^{(j)}}$. Thus, for $\alpha=(\alpha^{(j)})$, one has
$K \subseteq \coprod\limits_{j\in J} K_{\alpha^{(j)}} =  K_\alpha$. This shows that 
$\{K_\alpha\}_{\alpha \in \mathbb{I}}$ is cofinal in $\mathscr{K}(X)$.

(S4) Let $\{\alpha_k\}_{k< \omega}$ be a non-decreasing sequence with
$\gamma = \sup\limits_{k < \omega} \alpha_k$ in $\mathbb{I}$. 
Since $\{p_\alpha^{(j)}\}_{\alpha \in \mathbb{I}_j}$ is a $\kappa$-sock on $X_j$,
one has 
$\lim\limits_k p^{(j)}_{\alpha_k^{(j)}} = p^{(j)}_{\gamma^{(j)}}$ in 
$\mathscr{C}(X_j,K^{(j)}_{\gamma^{(j)}})$ for every $j \in J$. 
Thus, $\lim\limits_k p^{(j)}_{\alpha_k^{(j)}} = p^{(j)}_{\gamma^{(j)}}$ 
in $\mathscr{C}(X_j,K_\gamma)$ for every $j \in J$,
and therefore
$\lim\limits_k p_{\alpha_k} = p_{\gamma}$ in
$\mathscr{C}(X,K_{\gamma})$, as desired.
\end{proof}

A map $f: X \rightarrow Y$ between 
spaces is a {\em $k$-covering} if for every $C\in \mathscr{K}(Y)$ 
there exists $K \in \mathscr{K}(X)$ such that $C \subseteq f(K)$
(\cite[p.~21]{McCoy}). 

\begin{proposition} \label{basic:prop:kw}
Let $f: X \rightarrow Y$ be a continuous map that is a  $k$-covering.
Then $kw(Y)\leq kw(X)$.
\end{proposition}

\begin{proof}
Let $C \in \mathscr{K}(Y)$. Since $f$ is a $k$-covering, there is $K\in \mathscr{K}(X)$
such that $C\subseteq f(K)$. One has $w(f(K)) \leq w(K)$ (\cite[3.3.7]{Engel}),
and therefore $w(C) \leq w(f(K)) \leq w(K) \leq kw(X)$, as desired.
\end{proof}

\begin{proposition}
Let $\kappa$ be an infinite cardinal,  and let $q \colon X\rightarrow Y$ 
be a quotient map that is a $k$-covering.\label{basic:prop:sock-quotient}
The following statements are equivalent for a 
$\kappa$-sock $\{p_\alpha\}_{\alpha \in \mathbb{I}}$ on $X$:

\begin{enumerate}[label={\rm (\roman*)}]

\item 
for every $\alpha \in \mathbb{I}$, $qp_\alpha$ is constant on the fibres of $q$; 

\item
there is a unique  $\kappa$-sock $\{\widetilde p_\alpha\}_{\alpha \in \mathbb{I}}$  on $Y$ such that
$\widetilde p_\alpha q = q p_\alpha$ for every $\alpha \in \mathbb{I}$.
\begin{align}
\bfig
\square|alra|/->`->`->`-->/[X`X`Y`Y;p_\alpha`q`q`\exists ! \widetilde p_\alpha]
\efig
\end{align}

\end{enumerate}

\end{proposition}

\begin{proof}
It is clear that (ii) implies (i), and so we prove only that (i) implies (ii).
Since  for every $\alpha\in\mathbb{I}$,
$q p_\alpha\colon X \rightarrow Y $ is constant on the fibres of $q$ 
(which is a quotient map), $q p_\alpha$ induces
a unique continuous map $\widetilde p_\alpha\colon Y \rightarrow Y$ such that
$\widetilde p_\alpha q = q p_\alpha$. We show that 
$\{\widetilde p_\alpha\}_{\alpha \in \mathbb{I}}$ is a $\kappa$-sock on 
$Y$.

(S1) The image
\begin{align}
\widetilde K_\alpha = \widetilde p_\alpha(Y) = \widetilde p_\alpha(q(X))
=q(p_\alpha(X)) = q(K_\alpha)
\end{align}
is compact, because $K_\alpha = p_\alpha(X)$ is compact.

(S2) Let $\alpha,\beta \in \mathbb{I}$ be such that $\alpha \leq \beta$. 
Since $\{p_\alpha\}_{\alpha \in \mathbb{I}}$ is a $\kappa$-sock on $X$,
one has $p_\beta p_\alpha = p_\alpha$. Thus,
\begin{align}
\widetilde p_\beta \widetilde p_\alpha q  = 
\widetilde p_\beta q p_\alpha = q p_\beta p_\alpha = q p_\alpha = \widetilde p_\alpha q,
\end{align}
and so $\widetilde p_\beta \widetilde p_\alpha = \widetilde p_\alpha$,
because $q$ is surjective. In particular, $\widetilde p_\alpha$ is a retract.

(S3) Let $C$ be a compact subset of $Y$. Since $q$ is a $k$-covering,
there is a compact subset $K$ of $X$ such that $C\subseteq  q(K)$.
There is $\alpha \in \mathbb{I}$ such that $K \subseteq K_\alpha = p_\alpha(X)$,
because $\{p_\alpha\}_{\alpha \in \mathbb{I}}$ is a $\kappa$-sock on~$X$. Thus,
\begin{align}
C \subseteq q(K) \subseteq q(p_\alpha(X)) = \widetilde p_\alpha (q(X)) = \widetilde p_\alpha(Y)=
\widetilde K_\alpha.
\end{align}

(S4) Let $\{\alpha_k\}_{k <\omega}$ be a non-decreasing sequence with
$\gamma = \sup\limits_{k < \omega} \alpha_k$ in $\mathbb{I}$.
Since $\{p_\alpha\}_{\alpha \in \mathbb{I}}$ is a $\kappa$-sock on $X$,
one has $\lim p_{\alpha_k} = p_\gamma$ in $\mathscr{C}(X,K_\gamma)$. Thus,
\begin{align}
    \lim \widetilde p_{\alpha_k} q = \lim q  p_{\alpha_k} = q p_\gamma
    = \widetilde p_\gamma q
    \label{eq:qpgamma}
\end{align}
in  $\mathscr{C}(X,q(K_\gamma)) = \mathscr{C}(X,\widetilde K_\gamma)$, 
because composition is continuous in the compact-open topology
(\cite[3.4.2]{Engel}). The natural map 
$q_*\colon \mathscr{C}(Y,\widetilde K_\gamma) \rightarrow \mathscr{C}(X,\widetilde K_\gamma)$
defined by $f \mapsto f q$ is not only continuous, but also an embedding,
because $q$ is a $k$-covering
(\cite[2.12(a)]{GLPHD}; see also \cite[2.2.6(c)]{McCoy}). By (\ref{eq:qpgamma}), 
$\lim q_*(\widetilde p_{\alpha_k}) = q_*(\widetilde p_\gamma)$, and therefore
$\lim \widetilde p_{\alpha_k} = \widetilde p_\gamma$ in $\mathscr{C}(Y,\widetilde K_\gamma)$, as desired.
\end{proof}

The {\itshape Long Line} $\mathbb{L}$ is obtained by gluing together two copies of the Closed Long Ray at the boundary points $0$. 
{\em Par abus de langage}, we refer to the image of the point $0$ in the quotient
as $0$ again.

\begin{corollary} \label{basic:cor:longline}
The Long Line $\mathbb{L}$  has a sock and
has a countable compact weight.
\end{corollary}

\begin{proof}
Let $X$ denote the disjoint union of two copies of the Closed Long Ray $\mathbb{L}_{\geq 0}$.
By definition, there is a quotient map $q\colon X \rightarrow \mathbb{L}$, whose fibres
are singletons with the exception of the point $0$, where $q^{-1}(\{0\})$ contains two points.
By Corollary~\ref{basic:cor:longray}, the Closed Long Ray $\mathbb{L}_{\geq 0}$ 
has a sock
defined by $x \mapsto \min(x,\alpha)$.
By Proposition~\ref{basic:prop:sock-coprod}, this induces 
a sock  $\{p_\alpha\}_{\alpha \in \mathbb{I}}$  on $X$, 
and it can easily be checked that $qp_\alpha$ are constant on the fibres of $q$.
Therefore, by Proposition~\ref{basic:prop:sock-quotient}, $\{p_\alpha\}_{\alpha \in \mathbb{I}}$ 
induces a unique sock  $\{\widetilde p_\alpha\}_{\alpha \in \mathbb{I}}$  on $\mathbb{L}$ such that $\widetilde p_\alpha q = q p_\alpha$ for every $\alpha \in \mathbb{I}$.
Lastly, by Proposition~\ref{basic:prop:kw}, $kw(\mathbb{L}) \leq kw(X) = \aleph_0$.
\end{proof}

\begin{proposition}
Let $\kappa$ be an infinite cardinal, \label{basic:prop:sock-prod}
let $\{X_j\}_{j \in J}$ be a family
of non-empty topological spaces, and put
$X \coloneqq \prod\limits_{j\in J} X_j$. Then:

\begin{enumerate}

\item 
$\sup\limits_{j \in J} kw(X_j)  \leq kw(X) \leq |J| \sup\limits_{j \in J} kw(X_j)$; and

\item
if $X_j$ has a $\kappa$-sock for every $j \in J$, then 
$X$ has a $\kappa$-sock.

\end{enumerate}
\end{proposition}

\begin{proof}
Let $\pi_j \colon X \rightarrow X_j$ denote the canonical projection for every $j \in J$. 

(a) Since $X_j$ is non-empty for ever $j\in J$, each $X_j$ can be embedded into the product
$X$, and so $kw(X_j) \leq kw(X)$.  On the other hand,  if $K \subseteq X$ is compact,
then $K \subseteq \prod\limits_{j \in J} \pi_j(K)$. Thus, by \cite[2.3.13]{Engel},
\begin{align}
w(K) \leq |J| \sup_{j \in J} w(\pi_j(K)) \leq 
|J| \sup_{j \in J} kw(X_j).
\end{align}

(b)  For every $j \in J$, let $\{p_\alpha^{(j)}\}_{\alpha \in \mathbb{I}_j}$ be a $\kappa$-sock on $X_j$. Put $\mathbb{I} \coloneqq \prod\limits_{j \in J} \mathbb{I}_j$ equipped
with the coordinatewise order, that is, $(\alpha^{(j)}) \leq (\beta^{(j)})$ if 
$\alpha^{(j)} \leq \beta^{(j)}$ for every $j \in J$. Then $(\mathbb{I},\leq)$ is a 
$\kappa$-long $\omega$-cpo, because upper bounds and suprema can be calculated 
coordinatewise. For $\alpha=(\alpha^{(j)})\in\mathbb{I}$, put
$p_\alpha\coloneqq \prod\limits_{j \in J} p^{(j)}_{\alpha^{(j)}}$. We verify that 
$\{p_\alpha\}_{\alpha \in \mathbb{I}}$ satisfies the conditions of a $\kappa$-sock on $X$.

(S1) For every $(\alpha^{(j)})\in\mathbb{I}$, the image
$K_\alpha\coloneqq p_\alpha(X)  =  
\prod\limits_{j \in J} p^{(j)}_{\alpha^{(j)}} (X_j)$ is compact,
because $p^{(j)}_{\alpha^{(j)}} (X_j)$ is compact for every $j \in J$.

(S2) If  $\alpha=(\alpha^{(j)}) \leq \beta=(\beta^{(j)})$ in $\mathbb{I}$, then
$\alpha^{(j)} \leq \beta^{(j)}$ for every $j \in J$, and so 
$p^{(j)}_{\beta^{(j)}} p^{(j)}_{\alpha^{(j)}} = p^{(j)}_{\alpha^{(j)}}$ for every $j \in J$.
Thus, $p_\beta p_\alpha = p_\alpha$.

(S3) If $K \subseteq X$ is compact, then $\pi_j(K)$ is compact in $X_j$ for 
every $j \in J$, and so there is $\alpha^{(j)} \in \mathbb{I}_j$ such that
$\pi_j(K) \subseteq K_{\alpha^{(j)}}$. Thus, for $\alpha=(\alpha^{(j)})$, one has
$K \subseteq \prod\limits_{j\in J} K_{\alpha^{(j)}} =  K_\alpha$. This shows that 
$\{K_\alpha\}_{\alpha \in \mathbb{I}}$ is cofinal in $\mathscr{K}(X)$.

(S4) Let $\{\alpha_k\}_{k< \omega}$ be a non-decreasing sequence with
$\gamma = \sup\limits_{k < \omega} \alpha_k$ in $\mathbb{I}$. 
Since $\{p_\alpha^{(j)}\}_{\alpha \in \mathbb{I}_j}$ is a $\kappa$-sock on $X_j$,
one has 
$\lim\limits_k p^{(j)}_{\alpha_k^{(j)}} = p^{(j)}_{\gamma^{(j)}}$ in 
$\mathscr{C}(X_j,K^{(j)}_{\gamma^{(j)}})$ for every $j \in J$. Thus,
$\lim\limits_k \pi_j p_{\alpha_k} = \pi_j p_{\gamma}$ in
$\mathscr{C}(X,K^{(j)}_{\gamma^{(j)}})$ for every $j \in J$, and therefore
$\lim\limits_k p_{\alpha_k} = p_{\gamma}$ in
$\mathscr{C}(X,K_{\gamma})$, as desired.
\end{proof}

Proposition~\ref{basic:prop:sock-prod} combined with Corollaries~\ref{basic:cor:delta}, \ref{basic:cor:longray}, and \ref{basic:cor:longline} yield the following conclusion.

\begin{corollary}\label{basic:cor:long}
Let $C$ be a compact space. For every $n,m,l<\omega$,
the product space 
\begin{align}
\mathbb{L}^n \times (\mathbb{L}_{\geq 0})^m \times (\omega_1)^l \times C
\end{align}
has a sock. Furthermore, if $C$ is metrizable, then 
$\mathbb{L}^n \times (\mathbb{L}_{\geq 0})^m \times (\omega_1)^l \times C$ has
a countable compact weight. \qed
\end{corollary}

It is well known that every continuous function from the Closed Long Ray $\mathbb{L}_{\geq 0}$ 
into the real line is eventually constant; however, only spaces whose Stone-\v{C}ech 
remainder is a singleton can satisfy such a strong property. For example, not every
continuous real valued function on the Long Line $\mathbb{L}$ is eventually constant.
Nevertheless, spaces that have a sock satisfy the property that every continuous function into 
a metrizable space is determined by its value on some compact subset.

\begin{theorem} \label{basic:thm:evalpha}
Let $X$ be a locally compact space with a sock $\{p_{\alpha}\}_{\alpha \in \mathbb{I}}$
such that, in addition, $p_\alpha p_\beta = p_\alpha$ for every $\alpha \leq \beta$.
Let $g\colon X \rightarrow M$ be a continuous map into a metrizable space $M$. Then
there is $\alpha_0 \in \mathbb{I}$ such that $g = g p_{\alpha_0}$.
\end{theorem}

In order to prove Theorem~\ref{basic:thm:evalpha}, we need a lemma.

\begin{lemma} \label{basic:lemma:evconst}
Let $(\mathbb{I}, \leq)$ be a long $\omega$-cpo, and let $M$ be a metrizable space.
Suppose that $f\colon \mathbb{I}\rightarrow M$ is a map such that
$f(\sup \alpha_k) = \lim f(\alpha_k)$
whenever $\{\alpha_k\}$ is a non-decreasing sequence in
$\mathbb{I}$. Then $f$ is eventually constant, that is,
there is $\alpha_0 \in \mathbb{I}$ such that
$f(\beta)=f(\alpha_0)$ for every $\beta \geq \alpha_0$.
\end{lemma}

\begin{proof}
{\itshape Step 1.} Suppose that $M=[0,1]$ equipped with the real topology.
Put $f_+(\alpha)\coloneqq \sup\limits_{\gamma \geq \alpha} f(\gamma)$
and $f_-(\alpha)\coloneqq \inf\limits_{\gamma \geq \alpha} f(\gamma)$.
Since $f_+$ and $f_-$ are monotone and $\mathbb{I}$ is long, there is
$\alpha_0\in \mathbb{I}$ such that $f_+(\beta) = f_+(\alpha_0)$ and
$f_-(\beta) = f_-(\alpha_0)$ for every $\beta \geq \alpha_0$
(\cite[2.10]{RDGL1}). Put $a \coloneqq f_-(\alpha_0)$ and
$b \coloneqq f_+(\alpha_0)$. Clearly, $a\leq b$, and it suffices to show that $a=b$.

We define $\{\alpha_k\}_{k < \omega}$ recursively as follows. 
We pick $\alpha_{k+1} \geq \alpha_k$ such that $f(\alpha_{k+1}) > b - \frac{1}{k+1}$
if $k$ is even and $f(\alpha_{k+1}) < a + \frac{1}{k+1}$ if $k$ is odd. 
(Such $\alpha_{k+1}$
exists, because $f_-(\alpha_k) =a$ and $f_+(\alpha_k)=b$.)
Since $\mathbb{I}$ is an $\omega$-cpo, $\sup \alpha_k$ exists, and by our assumption,
$\lim f(\alpha_k) = f(\sup \alpha_k)$ also exists. Therefore, $b \leq a$, 
and $f(\beta) = a$ for every $\beta \geq \alpha_0$.

{\itshape Step 2.} Suppose that $M$ is a discrete space. Assume that $f$ is
not eventually constant. Then for every $\beta \in \mathbb{I}$ there is 
$\alpha \geq \beta$ such that $f(\alpha) \neq f(\beta)$. Thus,
we can construct recursively a non-decreasing sequence
$\{\alpha_k\}_{k < \omega}$ such that $f(\alpha_{k+1}) \neq f(\alpha_k)$.
Since $\mathbb{I}$ is an $\omega$-cpo, $\sup \alpha_k$ exists, and by our assumption,
$f(\sup \alpha_k) = \lim f(\alpha_k)$; however, $\lim f(\alpha_k)$ does not exist,
because $M$ is discrete and  $f(\alpha_{k+1}) \neq f(\alpha_k)$ for every $k$. 
This contradiction shows that $f$ is eventually constant. 

{\itshape Step 3.} Suppose that $M=J(\lambda)$, the hedgehog space of weight
$\lambda$, obtained by gluing together $\lambda$ many copies of $[0,1]$ at the point $0$,
equipped with the metric topology (see \cite[4.1.5]{Engel}). Let $d$ denote
the metric on $J(\lambda)$, and define $g\colon \mathbb{I}\rightarrow [0,1]$
by putting $g(\alpha) = d(f(\alpha),0)$. Since the metric is continuous,
$g(\sup \alpha_k) = \lim g(\alpha_k)$
whenever $\{\alpha_k\}$ is a non-decreasing sequence in
$\mathbb{I}$. Thus, by Step 1, there is $\alpha_0 \in \mathbb{I}$ such that
$g(\beta)=g(\alpha_0)$ for every $\beta \geq \alpha_0$. Put 
$r\coloneqq g(\alpha_0) = d(f(\alpha_0),0)$.
Let $D = \{x \in J(\lambda) \mid d(x,0) =r\}$. Then $D$ is discrete and
$f(\beta) \in D$ for every $\beta \geq \alpha_0$. Thus,  by
Step 2, $f$~is eventually constant.

{\itshape Step 4.} Let $M$ be any metrizable space. Put $\lambda = w(M)$.
By Kowalsky's Hedgehog Theorem, 
$M$ embeds as a subspace into $J(\lambda)^\omega$ (see \cite[4.4.9]{Engel}).
Let $\pi_n \colon J(\lambda)^\omega \rightarrow J(\lambda)$ denote the $n$-th projection.
By Step 3, each map $\pi_n f$ is eventually constant, and so there is $\alpha_n \in \mathbb{I}$
such that $\pi_n f(\beta) = \pi_n f(\alpha_n)$ for every $\beta \geq \alpha_n$.
Since $\mathbb{I}$ is long, there is $\gamma \in \mathbb{I}$ such that 
$\gamma \geq \alpha_n$ for every $n$. Therefore, one has $\pi_n f(\beta) = \pi_n f(\gamma)$
for every $\beta \geq \gamma$ and every $n$, that is, $f(\beta) = f(\gamma)$, as desired.
\end{proof}

We are now ready to prove Theorem~\ref{basic:thm:evalpha}.

\begin{proof}[Proof of Theorem~\ref{basic:thm:evalpha}.]
Let $d$ be a metric on $M$, and let $\bar d$ denote the associated
uniform metric on the function space $\mathscr{C}(X,M)$.
Consider $f\colon \mathbb{I} \rightarrow (\mathscr{C}(X,M),\bar d)$ defined by
$f(\alpha) =  gp_\alpha$. We show that $f$ satisfies the condition
of Lemma~\ref{basic:lemma:evconst}.

Let $\{\alpha_k\}$ be a non-decreasing sequence in $\mathbb{I}$, and put
$\gamma = \sup \alpha_k$. By property (S4), $\lim\limits_k p_{\alpha_k} = p_\gamma$
in $\mathscr{C}(X,K_\gamma)\subseteq \mathscr{C}(X,X)$ (where the latter is equipped with the compact-open topology).
Since composition is continuous with respect to the compact-open topology,
$\lim\limits_k g p_{\alpha_k} = g p_\gamma$ in $\mathscr{C}(X,M)$, and in particular,
$g p_{\alpha_k}$ converges uniformly to $g p_{\gamma}$ on $K_\gamma$. Therefore,
$f(\alpha_k) = g p_{\alpha_k} = g p_{\alpha_k} p_\gamma$ converges uniformly to
$g p_\gamma p_\gamma = g p_\gamma = f(\gamma)$. This shows that
$\lim f(\alpha_k) = f(\lim \alpha_k)$ in the metric space $(\mathscr{C}(X,M),\bar d)$.
Hence, by Lemma~\ref{basic:lemma:evconst}, there is $\alpha_0 \in \mathbb{I}$ such that
$f(\alpha) = f(\alpha_0)$ for every $\alpha \geq \alpha_0$.

Lastly, let $x \in X$. There is $\alpha \in \mathbb{I}$ such that $x \in K_{\alpha}$.
Since $\mathbb{I}$ is directed, without loss of generality, we may assume that
$\alpha \geq \alpha_0$. Thus,
\begin{align}
g(x) = g(p_\alpha(x)) = f(\alpha)(x) = f(\alpha_0)(x) = g(p_{\alpha_0}(x)),
\end{align}
as desired.
\end{proof}




\section{Homeomorphism groups of spaces with socks}

\label{sect:socks}

In  this section, we prove Theorems~\ref{main:thm:product} and~\ref{main:thm:sock}.
Theorem~\ref{main:thm:sock} combined with Corollary~\ref{basic:cor:long} yields
Theorem~\ref{main:thm:product}, so it suffices to prove Theorem~\ref{main:thm:sock}.
We start off with a well-known lemma, whose proof is provided here only for the sake
of completeness.

\begin{lemma}   \label{socks:lemma:tightness}
If $C$ and $K$ are compact spaces, then $t(\mathscr{C}(C,K)) \leq w(K)$.
\end{lemma}

\begin{proof}
Since $K$ is compact, its uniformity admits a base $\{U_i\}_{i\in I}$ of entourages
of the diagonal such that $|I| \leq w(K)$ (\cite[8.3.13]{Engel}).
The compact-open topology on $\mathscr{C}(C,K)$ is induced by uniform convergence
on $C$ (\cite[8.2.7]{Engel}), and has a base of the form $\{\widetilde U_i\}_{i\in I}$, 
consisting of entourages of the diagonal of the form
\begin{align}
\widetilde U_i \coloneqq \{(f,g) \in \mathscr{C}(C,K) \times \mathscr{C}(C,K) \mid
(f(x),g(x)) \in U_i \text{ for all } x \in C \}.
\end{align}
In particular, the topology of $\mathscr{C}(C,K)$ has a base of cardinality 
at most $|I|$ at each point, and thus the {\itshape character}
$\chi(\mathscr{C}(C,K))$ of the space $\mathscr{C}(C,K)$
is at most $w(K)$. Therefore,
\begin{align}
t(\mathscr{C}(C,K)) \leq \chi(\mathscr{C}(C,K)) \leq  w(K),
\end{align}
as desired.
\end{proof}

Recall that we denote by $\supp f \coloneqq \operatorname{cl}_X\{x\in X \mid f(x)\neq x\}$
the support of a homeomorphism of a space $X$.

\begin{Ltheorem*}[\ref{main:thm:sock}]
Let $\kappa$ be an infinite cardinal. Suppose that $X$ is a locally
compact space that has a $\kappa$-sock  and $kw(X) \leq \kappa$. Then $X$ has CSHP and
$\homeo_{cpt}(X)$ is $\kappa$-tight.
\end{Ltheorem*}

\begin{proof}
We first show that $\homeo_{cpt}(X)$ equipped with the topology induced
by $\homeo(\beta X)$ is $\kappa$-tight.
Let  $A \subseteq\homeo_{cpt}(X)$ and let $g \in \overline A$. Without loss of generality,
we may assume that $g=\operatorname{id}_X$. We construct $D\subseteq A$ such that
$|D|\leq \kappa$ and $\operatorname{id}_X \in \overline D$.

Let $D_{-1} \subseteq A$ be a non-empty subset such that $|D_{-1}| \leq \kappa$, and let $\alpha_{-1}\in\mathbb{I}$ be an arbitrary element. 
We construct inductively a non-decreasing family $\{D_m\}_{m < \omega}$ of 
subsets of $A$  containing $D_{-1}$ and a non-decreasing sequence  $\{\alpha_m\}_{m < \omega}$ in $\mathbb{I}$
such that for every $m \in \omega$:

\begin{enumerate}[label=(\Roman*),labelwidth=0.4in,leftmargin=0.4in]

\item 
$|D_m| \leq \kappa$;
\label{sock:thm:prop:D_m}

\item
$\supp f  
\subseteq K_{\alpha_m}$ for every $f\in D_{m-1}$; and
\label{sock:thm:prop:supp}

\item
$\beta p_{\alpha_{m}} \in \overline{(\beta p_{\alpha_{m}})^*(D_m)}$, where 
$(\beta p_{\alpha_{m}})^* \colon \mathscr{C}(\beta X,\beta X)\longrightarrow 
\mathscr{C}(\beta X, K_{\alpha_{m}})$ is the map defined by $h \mapsto \beta p_{\alpha_{m}} h$,
which is continuous (\cite[3.4.2]{Engel}).
\label{sock:thm:prop:closure}

\end{enumerate}

Suppose that  $D_m$ and $\alpha_m$ have been constructed.
Since the family $\{K_\alpha\}_{\alpha \in \mathbb{I}}$ is cofinal in $\mathscr{K}(X)$,
for every $f \in D_m$ there is $\alpha(f) \in \mathbb{I}$ such that
$\supp f \subseteq K_{\alpha(f)}$. One has
$|\{ \alpha(f) \mid f \in D_m\}| \leq |D_m| \leq \kappa$, and so
there is $\alpha_{m+1} \in \mathbb{I}$ such that $\alpha(f) \leq \alpha_{m+1}$
for every $f \in D_m$ and $\alpha_{m+1} \geq \alpha_{m}$, because $\mathbb{I}$
is $\kappa$-long. We observe that for every $f \in D_m$, one has
\begin{align}
    \supp f \subseteq K_{\alpha(f)} \subseteq K_{\alpha_{m+1}}.
\end{align}
One has 
\begin{align}
\beta p_{\alpha_{m+1}} = (\beta p_{\alpha_{m+1}})^*(\operatorname{id}_X) \in (\beta p_{\alpha_{m+1}})^*(\overline A)
\subseteq \overline{(\beta p_{\alpha_{m+1}})^*(A)}. \label{eq:pmalpha1}
\end{align}
By Lemma \ref{socks:lemma:tightness}, we have
\begin{align}
t(\mathscr{C}(\beta X, K_{\alpha_{m+1}}))\leq w(K_{\alpha_{m+1}}) \leq kw(X) \leq \kappa.
\end{align}
It follows from (\ref{eq:pmalpha1}) that there is $E_m \subseteq A$ such that
$|E_m| \leq \kappa$ and $\beta p_{\alpha_{m+1}} \in \overline{(\beta p_{\alpha_{m+1}})^*(E_m)}$. 
Put $D_{m+1} \coloneqq D_m \cup E_m$. It follows from the construction that $|D_{m+1}|\leq \kappa$.

Put $D \coloneqq \bigcup\limits_{m < \omega} D_m$ and $\gamma = \sup \alpha_m$. (Since
$\mathbb{I}$ is an $\omega$-cpo, $\gamma$ exists.)
It follows from {\ref{sock:thm:prop:D_m}} that $|D| \leq \aleph_0 \kappa =\kappa$.
We prove that $\operatorname{id}_X \in \overline D$ in $\homeo_{cpt}(X)$.
The topology of $\homeo_{cpt}(X)$  is induced by the topology
of $\mathscr{C}(\beta X,\beta X)$, which coincides with the uniform topology; 
as such, it is generated by pseudometrics of the form
\begin{align}
\bar d (f,g) \coloneqq \sup\{d(f(x),g(x)) \mid x \in X \},
\end{align}
where $d$ is a pseudometric in $\beta X$. Thus, it suffices to show
that for every pseudometric $d$ on $\beta X$ and $\varepsilon > 0$, there
is $f \in D$ such that
$d(f(x),x) < \varepsilon$ for every $x \in X$.

Let $d$ be a pseudometric on $\beta X$ and let $\varepsilon > 0$.
Since $\{p_\alpha\}_{\alpha \in \mathbb{I}}$ is  a $\kappa$-sock, 
one has $\lim p_{\alpha_m} = p_\gamma$ in 
$\mathscr{C}(X,K_{\gamma})$; in particular, $\{p_{\alpha_m}\}_{m < \omega}$ converges
uniformly to $p_\gamma$ on the compact set $K_\gamma$, and so
there is $m \in \omega$ such that for every $x \in K_\gamma$,
\begin{align}
d (p_{\alpha_m}(x) ,  x) = 
    d (p_{\alpha_m}(x) , p_\gamma(x)) < \frac \varepsilon 3.
    \label{eq:sock1}
\end{align}
By property {\ref{sock:thm:prop:closure}},  $\beta p_{\alpha_{m}} \in \overline{(\beta p_{\alpha_{m}})^*(D_m)}$
in $\mathscr{C}(\beta X, K_{\alpha_{m}})$, which also carries the uniform topology,
and so there is $f \in D_m$ such that 
\begin{align}
    d(p_{\alpha_m}(f(x)),p_{\alpha_m}(x))  = 
    d(((\beta p_{\alpha_{m}})^*(f))(x),(\beta p_{\alpha_{m}})(x)) < \frac \varepsilon 3
    \quad \text{ for every } x\in X.
     \label{eq:sock2}
\end{align}
By property {\ref{sock:thm:prop:supp}}, $\supp f \subseteq K_{\alpha_{m+1}} \subseteq K_\gamma$,
and so $f(x) \in K_\gamma$ for every $x \in K_\gamma$. Therefore, by
(\ref{eq:sock1}) applied to $f(x)$ instead of $x$, one has 
\begin{align}
d (p_{\alpha_m}(f(x)) ,  f(x))  < \frac \varepsilon 3
\quad \text{ for every } x\in K_\gamma.
    \label{eq:sock3}
\end{align}

Hence, by (\ref{eq:sock1})--(\ref{eq:sock3}), for every $x \in K_\gamma$,
\begin{align}
d(f(x),x) \leq 
d(f(x), p_{\alpha_m}(f(x))) + 
d(p_{\alpha_m}(f(x)),p_{\alpha_m}(x)) + 
d (p_{\alpha_m}(x) ,  x) < \varepsilon.
\end{align}
Since $f(x)=x$ for $x \in X\backslash K_\gamma$ and $f\in D_m \subseteq D$, this 
completes the proof that $\operatorname{id}_X \in \overline D$.
This shows that $\homeo_{cpt}(X)$ is $\kappa$-tight.

Lastly, we prove that $X$ has CSHP. To that end,
we show that the family $\{\homeo_K(X)\}_{K\in \mathscr{K}(X)}$ satisfies the conditions
of \cite[2.3]{RDGL1}. First, since $\{K_\alpha\}_{\alpha \in \mathbb{I}}$ is $\kappa$-long and
cofinal in $\mathscr{K}(X)$, it follows that $\mathscr{K}(X)$ itself is $\kappa$-long. 
Second, the inclusion $\homeo_K(X) \rightarrow \homeo_{cpt}(X)$ is an embedding for every $K\in\mathscr{K}(X)$. Third,  we have just shown 
that $t(\homeo_{cpt}(X)) \leq \kappa$.
Hence, by \cite[2.3]{RDGL1}, the topology of $\homeo_{cpt}(X)$ coincides
with the colimit space topology, as desired. 
\end{proof}




\section*{Acknowledgments}

We are grateful to Karen Kipper for her kind help in proofreading this paper for grammar and punctuation.

{\footnotesize

\bibliography{dahmen-lukacs}

}

\begin{samepage}

\bigskip
\noindent
\begin{tabular}{l @{\hspace{1.6cm}} l}
Rafael Dahmen						    & G\'abor Luk\'acs \\
Department of Mathematics				& Department of Mathematics and Statistics\\
Karlsruhe Institute of Technology		& Dalhousie University\\
D-76128 Karlsruhe					    & Halifax, B3H 3J5, Nova Scotia\\
Germany                			        & Canada\\
{\itshape rafael.dahmen@kit.edu}        & {\itshape lukacs@topgroups.ca}
\end{tabular}

\end{samepage}

\end{document}